\documentclass[preprint,12pt]{elsarticle}
\usepackage{amsmath}
\usepackage{amssymb} \usepackage{amsthm} \usepackage{enumerate}
\usepackage[english]{babel}
\usepackage{xcolor}

\newtheorem{Theorem}{Theorem}[section]
\newtheorem{Corollary}[Theorem]{Corollary}
\newtheorem{Lemma}[Theorem]{Lemma}

\newtheorem{Proposition}[Theorem]{Proposition}
\theoremstyle{definition}
\newtheorem{Definition}[Theorem]{Definition}
\newtheorem{Notation}[Theorem]{Notation}

\newtheorem{Remark}[Theorem]{Remark}

\def\F{{\mathbb F}}

\def\A{{\cal A}}

\def\L{{\cal L}}

\def\B{{\cal B}}

\newcommand{\cccell}[3]{\colorbox{#1}{\hbox to #2{\hfil #3\hfil}}}
\newcommand{\lccell}[3]{\colorbox{#1}{\hbox to #2{#3\hfil}}}
\newcommand{\rccell}[3]{\colorbox{#1}{\hbox to #2{\hfil #3}}}

\newcommand{\s}{\mathcal{S}}

\journal{Journal of Pure and Applied Algebra}
\begin{document}
\begin{frontmatter}

\title{Algebras of length one}

\author[MSU,MCFAM,MIPT]{O.V.~Markova\corref{cor1}}
\ead{ov\_markova@mail.ru}
\author[DMUO]{C.~Mart\'inez}
\ead{cmartinez@uniovi.es}
\author[DMUFC,DMUSP]{R.L.~Rodrigues}
\ead{rodrigo@mat.ufc.br}

\cortext[cor1]{Corresponding author}

\address[MSU]{Department of Mechanics and Mathematics, Lomonosov Moscow State University, Moscow, 119991, Russia.}
\address[MCFAM]{ Moscow Center
for Fundamental and Applied Mathematics, Moscow, 119991, Russia.}
\address[MIPT]{Moscow Institute of Physics and Technology, Dolgoprudny, 141701, Russia.}
\address[DMUO]{Department of Mathematics, University of Oviedo, Calvo Sotelo, s/n, 33007 Oviedo, Spain.}
\address[DMUFC]{Department of Mathematics, Federal University of Cear\'a, Pici Campus, Block 914, 60455-760, Fortaleza, Brazil.}
\address[DMUSP]{Department of Mathematics, University of S\~ao Paulo, Butant\~a, 05508-090, S\~ao Paulo, Brazil.}

\begin{abstract}

{Let $\A$ be an algebra over a field $\F$ and let $\s$ be a generating set of $\A$. The length of $\s$ indicates the maximal length needed to express an arbitrary element of $\A$ as a linear combination of words in the elements of $\s$. The length of an algebra $\A$ is defined as the maximum of lengths of its generating sets. In this paper (not necessarily associative) algebras, over fields of arbitrary characteristic, having length equal to 1  are determined.}
\end{abstract}

\end{frontmatter}

\section*{Introduction}

In the present paper, we will denote by $\A$ a unital (with identity element $1_{\A}$) not necessarily associative algebra over a field $\F$.
Information about non-associative algebras can be found, for instance, in  \cite{Osb72, Schaf, ZheSSS}.

Any product of a finite number of elements of a finite subset $\s \subset \A$ is a word in the letters from $\s$. The length of a word is equal to the number of letters different of $1_{\A}$ in the corresponding product. By convention, $1_{\A}$ is a word  in any subset $ {\s}$
of length 0.

If $\s$ is a generating system of the algebra $\A$, that is, $\A$ is the smallest subalgebra of $\A$ containing $\s$, then any element of $\A$ can be expressed as a linear combination of words in elements of $\s$.
If we can express all elements of $\A$ using words of length at most $k$, but we can not use only words of length at most $k-1$, we say that the length of the generating system is $k$. The length of a finitely generated algebra $\A$, $l(\A)$, is the maximal length of its finite generating systems.

Notice that the unital algebra $\A$ has length equal to 0 if and only if $\A = \F 1_{\A}$. Otherwise length is a positive integer or infinite.

If $\A$ is not unital, we can consider the unital hull of $\A$, $\tilde{\A} = \F 1_{\A} + \A$, and clearly $l(\A) = l(\tilde{\A})$.

A finite-dimensional associative algebra has always finite length and a finitely generated associative algebra with finite length has finite dimension.  Algebras considered in this paper will be always finite-dimensional.

Guterman and Kudryavtsev in \cite{GutK17, GutK20} started the study of length of non-associa\-tive algebras. In particular, they showed that any finite-dimensional non-associative algebra $\A$ has finite length, and if  $\dim \A = n > 2$,  then the length is bounded by $l(\A) \leq 2^{\dim \A-2}$.

The length of an algebra is an important invariant for the study of finite-dimensional algebras.   In some sense it measures the multiplicative complexity of the algebra.  The knowledge of a universal upper bound for the length gives the maximal size of products in generators that we need to consider in order to verify some property.
First of all, the length of a subset in a given algebra  determines the complexity of the procedures  verifying whether this  set generates the algebra and computing a basis of the algebra from the generating set. For matrix algebras, properties under consideration include simultaneous reduction (by similarity) of a set of matrices to a canonical, for example, (block-)triangular form \cite{Alp2}, or unitary similarity \cite{Alp1}. That is, the length function has applications in computational methods of matrix theory.

Therefore, the study of algebras whose length is close to the minimal value is of interest. Clearly the minimal non-trivial value of the length is $1$. Matrix algebras of length $1$ were described in \cite{Mar12}. The complete classification of associative algebras of length $1$  is also known (see \cite{Mar13}).

A similar question for non-associative algebras had not been addressed until now, may be because the problem, in a totally general context, seems somehow wild.
It was unclear what kind of tools should be used. In this paper we  address this question and give a complete description of algebras of length $1$ over an arbitrary field $\F$, mainly using methods of linear algebra and linearization techniques.

Our main result gives a characterization of algebras of length 1 in terms of a basis whose multiplication table satisfies some precise conditions.

\section{Preliminaries}

In order to place the problem in the proper context and to make it easier to readers we will start with some general definitions and results.

\begin{Notation}   Given $\s \subseteq \A$, $\F \s$ (or $\F(a_1, \ldots, a_n)$, if $\s = \left\{a_1, \ldots, a_n\right\}$) denotes the vector space linearly spanned by the elements of $\s$ and ${\rm alg} (\s)$ (or ${\rm alg}(a_1, \ldots, a_n)$, if $\s = \left\{a_1, \ldots, a_n\right\}$) will represent the subalgebra of $\A$ generated by $\s$, that is, ${\rm alg} (\s) =  \cap \ \B$, where $\B$ runs over all subalgebras of $\A$  with $\s \subseteq \B$.

The set of all words in $\mathcal{S}$ with length
 not greater than $i$ is denoted by $\s^{i}$, for $i \geq 0$ and the set $\L_{i}(\s)$ stands for $\F \s^{i} $. Clearly  $\displaystyle \bigcup_{i \geq 0} \mathcal{L}_{i} (\s) = {\rm alg} (\s)$.
\end{Notation}

\begin{Remark}
A subset $\s \subset \A$ is a generating set of $\A$ if and only if $\A  = {\rm alg} (\s)$.
\end{Remark}

With these definitions and notations in mind we may now introduce formally the concept of length, both for  $\s$ and for $\A$,  as follows.

\begin{Definition}
The length of a  finite subset $\s$ of $\A$ is
$$l(\s) = {\rm min}\left\{k \in \mathbb{Z}_{+} : {\rm alg} (\s) = \L_{k}(\s) \right\}.$$
\end{Definition}

\begin{Definition}
The length of a finitely generated algebra $\A$ is
$$l(\A) = {\rm max}\left\{l(\s) :  \A = {\rm alg} (\s) \hspace{0.5cm}  \s \ {\rm finite}   \right\}.$$
\end{Definition}

Given a set  $\s$  its length evaluation is a standard linear algebraic problem: the construction of a  basis of  ${\rm alg} (\s)$. But  its computational complexity comes from the fact that the number of  words in each  $\s^i$ grows exponentially. We would like to emphasize that  the different behaviour  in associative and non-associative cases is linked to the different  properties of the sequence $\{\dim \L_i(\s)\},\ i=0,1,\ldots$. In the  associative case  this sequence is strictly increasing, that is,   $\dim \L_i(\s)< \dim \L_{i+1}(\s)$, until it stabilizes at  step  $l(\s)$, that is, $\L_{l(\s)}(\s)={\rm alg} (\s)$ and $\dim \L_n(\s)=\dim {\rm alg} (\s)$ for every $n\geq l(\s)$. In the   non-associative case the sequence   $\{\dim \L_i(\s)\},\ i=0,1,\ldots$ is non-decreasing, that is,   $\dim \L_i(\s)\leq \dim \L_{i+1}(\s)$ until it stabilizes. But  in the non-associative case the equality $ \L_i(\s)=\L_{i+1}(\s)$ does not imply that  $\L_i(\s)={\rm alg} (\s)$.  However, for a given $n\geq 1$, it is sufficient to have   \[\dim \L_{n}(\s) = \dim \L_{n+1}(\s) = \ldots = \dim \L_{2n}(\s)\]  to guarantee that the identity $\dim \L_{n}(\s) = \dim \L_{n+t}(\s)=\dim {\rm alg} (\s)$ holds for any $t \geq 1$,  as it was established in \cite[Proposition 2.3]{GutK20}. Consequently, the equality $ \L_1(\s)=\L_{2}(\s)$  implies  that  $\L_{1}(\s) = \L_{t}(\s)={\rm alg} (\s)$  for all $t \geq 2$.

Applications of the length function appear, for example, in computational methods of matrix theory, since the length determines the complexity of some rational procedures. Al'pin and Ikramov extended the classical Specht criterion for  unitary similarity between two complex $n \times n$ matrices to  unitary similarity between two normal matrix sets of cardinality $m$. In the same way, the well-known result of Pearcy extending Specht's theorem can be used as a finiteness criterion. The complexity of this criterion depends both, on $n$ and on the length of  algebras under their analysis \cite{Alp1, Alp2}.

The problem of computing  lengths of  the matrix algebra $M_n(\F)$, as a function of the matrix size $n$,  was first addressed by Paz~\cite{Paz84} in 1984, and still remains open. The known bounds are non-linear functions of $n$, see \cite{Paz84, Pap97, Shi19}.  Linear bounds for specific generating sets has been found in \cite{GLMSh18, GutMM09, GutMM19}. See also references therein.

In the case of arbitrary associative algebras Pappacena \cite{Pap97} provided an upper bound for the length of any finite-dimensional associative algebra $\A$ as a function of two invariants of $\A$: its dimension and $m(\A)$, the maximal degree of minimal polynomials of elements in the algebra. In \cite{Mar09SM} a stronger bound, depending of the same two parameters, was obtained for commutative algebras. A study of algebraic properties of the length function, namely the length of the direct sum, can be found in \cite{GutKM19}.

Let us remember the classification of associative algebras of length 1.

\begin{Theorem}[{\cite[Theorem 3]{Mar13}}]\label{markova}
Let $\F$ be an arbitrary field and let $\A$ be a finite-dimensional associative unital $\F$-algebra. Then $l(\A) = 1$ if and only if it is one of the algebras in the following list:
\begin{itemize}
\item[$(1)$] $\A \cong \F \oplus \F$;
\item[$(2)$] $\A$ is a field, $\dim \A = 2$;
\item[$(3)$] $\F = \F_{2}$, $\A \cong \F \oplus \F \oplus \F$;
\item[$(4)$] {\rm dim} $\A/J(\A) = 1$, $J(\A) \neq 0$, $J(\A)^{2} = 0$, where $J(\A)$ denotes  the Jacobson radical of $\A$;
\item[$(5)$] There exist elements $e, f \in \A$ such that
$e^{2} = e, f^{2} = f, ef = fe = 0, e + f = 1; {\rm dim} \ e \A e = {\rm dim} \ f \A f = 1; f \A e = 0,  e \A f \neq 0.$
\end{itemize}
Algebras of different types are not isomorphic. Two algebras $\B, \mathcal{C}$ of type $(4)$  (respectively of type $(5)$) are isomorphic if and only if ${\rm dim} \ \B = {\rm dim} \ \mathcal{C}$.
\end{Theorem}

From now on we will not assume associativity of the considered algebras.  Our aim is to determine algebras of length 1, without any additional assumption.

\section{General results}

In this section we give some general results about algebras of length 1. The main result of the section gives a characterization of them that is essential for the rest of results in this paper. Some consequences of this result are given explicitly. In particular, algebras of length 1 are proved to be power-associative.

To start, let us  remark that $l(\A) = 1$ means that if $\s$ is  an arbitrary generating  set of  ${\A}$  and $\bar{\s} = \s \cup \{1_{\A}\}$, then  $\bar{\s}$  linearly spans $\A$, that is, $\A = \rm{alg}(\s)$ implies $\A = \F \bar{\s}$. If $\A$ is a finitely generated algebra over a field $\F$ and $l(\A) = 1$, then $\A$ is finite-dimensional.

Next proposition will play the main role to get our description of algebras of length 1 in terms of the existence of a basis with a known multiplication table.

\begin{Proposition}\label{proposition0}
Let $\A$ be an $\F$-algebra of dimension $n$ with identity element $1_{\A}$. $\A$ has length $1$ if and only if $ab \in \F(1_{\A}, a, b)$, for all $a, b \in \A$.
\end{Proposition}
\begin{proof}
Let $\A$ be an $\F$-algebra of length 1. We will prove first that $a^{2} \in \F(1_{A}, a)$. If $a \in \F (1_{\A})$, the result is obvious. Therefore we can assume w.l.g. that $\left\{1_{\A}, a\right\}$ is a linearly independent set, and extend it to a basis $\left\{s_1 = a, s_2 = 1_{\A},  s_3, \ldots, s_n\right\}$ of $\A$. So $a^2 = \alpha_1 a + \alpha_2 1_{\A} + \cdots + \alpha_n s_n$ for some $\alpha_1, \ldots, \alpha_n \in \F$. If $a^2 \notin \F (1_{\A}, a)$, then there is some $\alpha_i \neq 0$, $i > 2$. We can assume (reordering the basis if needed) that $\alpha_3 \neq 0$. This implies that $s_3 \in  {\rm alg}(s_1 = a, s_2 = 1_{\A}, s_4, \ldots, s_n)$. Consequently $\A = {\rm alg}(\s') \neq \F  \s' $, where $\overline{\s'} = \s' = \s \setminus \left\{s_3\right\}$, a contradiction.

If $a, b \in \A$, let us consider $\left\{1_{\A}, a, b\right\}$. For $\left\{1_{\A}, a, b\right\}$ linearly dependent the result directly follows from what has been proved above. So we can assume that $\left\{1_{\A}, a, b\right\}$ are linearly independent and extend this set (as we did before) to a basis $\left\{1_{\A}, a, b, s_4, \ldots, s_n\right\}$. Then $ab = \lambda_1 1_{\A} + \lambda_2 a + \lambda_3 b + \lambda_4 s_4 + \cdots + \lambda_n s_n$ for some $\lambda_1, \lambda_2, \lambda_3, \lambda_4, \ldots, \lambda_n \in \F$. If there is some $\lambda_i \neq 0$, $i \geq 4$ (we can assume $\lambda_4 \neq 0)$, then considering $\s'' = \s \setminus \left\{s_4\right\}$ we have, as before, $\A = {\rm alg} (\s'') \neq \F \s''$, a contradiction.

The converse is clear.
\end{proof}

\begin{Corollary}\label{lemma0}
Let $\A$ be a finite dimensional algebra with identity element $1_{\A}$ over a field $\F$. If $l(\A) = 1$, then $a^{2} \in \F(1_{\A}, a)$, for all $a \in \A$.
\end{Corollary}

Next results  immediately follow from Proposition \ref{proposition0}.

\begin{Corollary}
Every unital algebra of dimension $2$ over an arbitrary field has length $1$.
\end{Corollary}

Let us notice that any 2-dimensional unital algebra is associative.

A power-associative algebra over a field $\F$ is an algebra in which every subalgebra generated by a single element is associative.  Examples include associative, Lie, Jordan  and  alternative algebras.

Since the unital subalgebra generated by an element of an algebra $\A$ with $l(\A) = 1$ has dimension at most 2, it is associative. Now next corollary immediately follows.

\begin{Corollary}
An algebra $\A$  of length $1$  over a field $\F$ is power-associative.
\end{Corollary}

\begin{Corollary}\label{2.5}
Let $\mathcal{B}$ be a subalgebra of a unital finite-dimensional $\F$-algebra $\A$. If $l(\A) = 1$, then $l(\mathcal{B}) = 1$.
\end{Corollary}

Notice that the above Corollaries \ref{lemma0} and \ref{2.5} extend results of Corollary 2.2 and Lemma 2.1 of \cite{Mar12} from associative case.

In general, it is not true that $l(\B) \leq l(\A)$ for any $\B$ subalgebra of $\A$, even in the associative case (see \cite{GutMM09}).

\section{Characteristic different from 2}

In this section we will address the aim of the paper in the case of characteristic different from 2. So in this section $\F$ will denote a field with ${\rm char}(\F) \neq  2$.

\begin{Lemma}\label{lemma3.1}
Let $\A$ be an $\F$-algebra  of dimension $n$ and length $1$ with identity element $1_{\A}$. Then
there is a basis $\B$ of $\A$ such that $b^{2} \in \F 1_{\A}$, for every $b \in \B$.
\end{Lemma}
\begin{proof}
Let $\left\{a_{1} = 1_{\A}, a_{2}, \ldots , a_{n}\right\}$ be a basis of $\A$. According to Corollary \ref{lemma0}, we can write $a_{i}^{2} = \alpha_{i}1_{\A} + \gamma_{i}a_{i}$, for some $\alpha_{i}, \gamma_{i} \in \F$, for all $2 \leq i \leq n$. It is easy to check that the basis $\B = \left\{a_{1} = 1_{\A}, a_{2} - \frac{1}{2} \gamma_{2} 1_{\A}, \ldots , a_{n} - \frac{1}{2} \gamma_{n} 1_{\A}\right\}$ satisfies the required conditions.
\end{proof}

\begin{Definition}
A basis $\B = \left\{1_{\A}, a_2, \ldots, a_n\right\}$ of a unital algebra $\A$ satisfying that $a_{i}^2 = \mu_i 1_{\A}$ for some $\mu_i \in \F$, $2 \leq i \leq n$ will be called canonical basis.
\end{Definition}

Notice that Lemma \ref{lemma3.1} says that an algebra of length 1 has some canonical basis.

\begin{Lemma}\label{3.2}
Let $\A$ be an $\F$-algebra of length  1 with identity element $1_{\A}$ and  let $ \B =\left\{a_{1} = 1_{\A}, a_{2},
a_{3}, \ldots, a_{n}\right\}$  be a canonical basis of $\A$. Then  $a_{i}a_{j} + a_{j}a_{i} \in \F 1_{\A}$, for every $2 \leq i \neq j \leq n$.
\end{Lemma}
\begin{proof}
Since $\B$ is a canonical basis, $a_{i}^2 = \mu_i 1_{\A}$, for some $\mu_i \in \F$, $2 \leq i \leq n$.

By Proposition \ref{proposition0}, $a_{i}a_{j} \in \F(1_{\A}, a_i, a_j)$  for every  pair of indices $i \neq j \in \{2, \ldots, n \}$.  So, there are elements $\alpha_{ij}, \beta_{ij}, \beta_{ji}^{*},  \alpha_{ji}, \beta_{ji}, \beta_{ij}^{*} \in \F$, $2 \leq i < j \leq n$, such that  $a_{i}a_{j} = \alpha_{ij}1_{\A} + \beta_{ij}a_{i} + \beta_{ji}^{*}a_{j}$  and  $a_{j}a_{i} = \alpha_{ji}1_{\A} + \beta_{ji}a_{j} + \beta_{ij}^{*}a_{i}$.

By Corollary \ref{lemma0} we know that for an arbitrary $\lambda \in \F$, $(a_{i} + \lambda a_{j})^{2} \in \F(1_{\A}, a_{i} + \lambda a_{j}).$ Since $(a_{i} + \lambda a_{j})^{2}=  a_i^2 + \lambda a_i a_j + \lambda a_j a_i + \lambda^2a_j^2 =$
\begin{align*}
=( \mu_{i} + \lambda^{2} \mu_{j} ) 1_{\A} + \lambda (\alpha_{ij} 1_{\A} + \beta_{ij}a_{i} + \beta_{ji}^{*}a_{j}) + \lambda(\alpha_{ji}1_{\A} + \beta_{ji}a_{j} + \beta_{ij}^{*}a_{i}),
\end{align*}
we conclude that the coefficient of $a_j$ is $\lambda$ times the coefficient of $a_i$, that is, $\lambda (\beta_{ji}^{*} + \beta_{ji}) = \lambda^{2} (\beta_{ij} + \beta_{ij}^{*}).$

Since $|\F| \geq 3$  and every $\lambda \in \F$ is a root of  $x (\beta_{ji}^{*} + \beta_{ji}) = x^2(\beta_{ij} + \beta_{ij}^{*}),$  we conclude that  $\beta_{ji}^{*} + \beta_{ji}  = 0$ for every distinct $i, j \in \{2, \ldots, n \}$.  That is,  $a_{i}a_{j} = \alpha_{ij}1_{\A} + \beta_{ij}a_{i} - \beta_{ji}a_{j}$, $a_{j}a_{i} = \alpha_{ji}1_{\A} + \beta_{ji}a_{j} - \beta_{ij}a_{i}$.

This implies that  $a_{i}a_{j} + a_{j}a_{i} = (\alpha_{ij} + \alpha_{ji})1_{\A}$,  and completes the proof.
\end{proof}

We will consider first the case ${\rm dim} \ \A = 3$, since it behaves in a different way.

\begin{Proposition} Let $\A$ be a unital  $\F$-algebra of dimension $3$.
If there is a canonical basis $\left\{a_{1} = 1_{\A}, a_{2}, a_{3}\right\}$ such that $a_{2}a_{3} + a_{3}a_{2} \in \F 1_{\A}$ then $l(\A) = 1$.
\end{Proposition}
\begin{proof}
Let $a$ and $b$ be arbitrary elements of $\A$. If $\s = \left\{1_{\A}, a, b\right\}$ is a linearly independent set, then it is a basis of $\A$ and so $ab \in \F \s = \F(1_{\A}, a, b)$.

Otherwise, there are scalars $\alpha, \beta, \gamma \in \F$, not all of them equal to zero, such that $\alpha 1_{\A} + \beta a + \gamma b = 0$.

If   $\gamma = 0$, then $\alpha 1_{\A} + \beta a = 0$, that is, $a \in \F 1_{\A}$. This implies that
$ab \in \F b \subseteq \F(1_{\A}, b)$.

If  $\gamma \neq 0$,  $b \in \F(1_{\A}, a)$, so $ab \in \F(a, a^{2})$. So, in order  to prove that $l(\A) = 1$, we only need to show that $a^{2} \in \F(1_{\A}, a)$, for every $a \in \A$.

Given  $a = \lambda_{1}1_{\A} + \lambda_{2}a_{2} + \lambda_{3}a_{3} \in \A$, using that $\B$ is a canonical basis it follows that $a^2 =$
$$=\lambda_{1}(\lambda_{1}1_{\A} + \lambda_{2}a_{2} + \lambda_{3}a_{3}) + \lambda_{2}a_{2}(\lambda_{1}1_{\A} + \lambda_{2}a_{2} + \lambda_{3}a_{3} ) + \lambda_{3}a_3(\lambda_{1}1_{\A} + \lambda_{2}a_{2} + \lambda_{3}a_{3})$$
\begin{flushleft}
$=2\lambda_{1}(\lambda_{1}1_{\A} + \lambda_{2}a_{2} + \lambda_{3}a_{3}) - \lambda_{1}^{2}1_{\A} + \lambda_{2}^{2}a_{2}^{2} + \lambda_{3}^{2}a_{3}^2 + \lambda_{2}\lambda_{3}(a_{2}a_{3} + a_{3}a_{2})$
\vspace{0.5cm}

$=2\lambda_{1}a - \lambda_{1}^{2}1_{\A} + \lambda_{2}^{2}a_{2}^{2} + \lambda_{3}^{2}a_{3}^2 + \lambda_{2}\lambda_{3}(a_{2}a_{3} + a_{3}a_{2})   \in \F(1_{\A}, a)$,
\end{flushleft}
since $a_{2}a_{3} + a_{3}a_{2}, a_2^2, a_3^2 \in \F 1_{\A}$.
\end{proof}

\begin{Corollary} Let $\A$ be a $3$-dimensional  $\F$-algebra with identity element $1_{\A}$. Then $l(\A) = 1$ if and only if  there is a canonical basis $\left\{a_{1} = 1_{\A}, a_{2}, a_{3}\right\}$ of $\A$ such that $a_2 a_3 + a_3 a_2 \in \F 1_{\A}$.
\end{Corollary}

\begin{Definition}\label{definition3.6}
Let $\A$ be a unital $\F$-algebra. A canonical basis of $\A$, $\B = \left\{1_{\A}, a_2, \ldots, a_n\right\}$ $(a_i^2 = \mu_i 1_{\A}, \mu_i \in \F)$ will be called special if there are elements $\alpha_{ij}, \beta_i \in \F$, for every $2 \leq i \neq j \leq n$ satisfying $a_i a_j = \alpha_{ij} 1_{\A} + \beta_j a_i - \beta_i a_j$.
\end{Definition}

We have proved in Lemma \ref{lemma3.1} that a unital algebra $\A$ of length 1 has always a canonical basis. Furthermore, Lemma \ref{3.2} gives a way to get a canonical basis from any basis of $\A$ that contains $1_{\A}$. Next proposition will show that an arbitrary canonical basis of $\A$ is special.

\begin{Proposition}\label{theorem3.5}
Let $\A$ be a unital algebra of length $1$. Then an arbitrary canonical basis $\B = \left\{1_{\A}, a_2, \ldots, a_n\right\}$ is special.
\end{Proposition}
\begin{proof}
Assume  $l(\A)=1$. The case ${\rm dim}(A) \leq 3$ is obvious, by Lemma \ref{3.2}. So we will assume ${\rm dim}(\A) \geq 4$. Using Lemma \ref{3.2} we know the existence of scalars $\mu_i$, $i \in \left\{2, \ldots, n\right\}$ and $\alpha_{ij}, \beta_{ij}$ for arbitrary distinct indices $i, j \in \left\{2, \ldots, n\right\}$ such that $a_{i}^2 = \mu_i 1_{\A}$, $a_i a_j = \alpha_{ij} 1_{\A} + \beta_{ij} a_{i} - \beta_{ji} a_{j}$.

Consider arbitrary  distinct indices $i, j, k \in$ $\left\{2, \ldots, n\right\}$.  Then for every  $\mu \in \F$ we have $$a_i(a_j+\mu a_k)\in \F(1_{\A}, a_i, a_j+ \mu a_k).$$
 As we have argued above, from  $a_i(a_j+ \mu a_k)=$ $$=\alpha_{ij}1_{\A} +\beta_{ij}a_i-\beta_{ji}a_j+ \mu \alpha_{ik}1_{\A}+ \mu \beta_{ik}a_i- \mu \beta_{ki}a_k= \alpha 1_{\A}+ \beta a_i + \gamma (a_j+ \mu a_k),$$
it follows that the coefficient of $a_k$ is equal to $\mu$ times the coefficient of $a_j$, that is, $\beta_{ki}=\beta_{ji}$, for $i, j, k$ distinct indices. Taking $\beta_i = \beta_{ki}$ for any $k \neq i$ in $\left\{2, \ldots, n \right\}$ we get the result.
\end{proof}

We can already prove the main theorem of the paper, that says that algebras of length 1 are characterized by the existence of an special algebra.

\begin{Theorem}\label{main}
Let $\A$ be a unital finite-dimensional $\F$ algebra. Then $l(\A) = 1$ if and only if there is a special basis $\B = \left\{1_{\A}, a_2, \ldots, a_n\right\}$ of $\A$.
\end{Theorem}
\begin{proof}
We already know the existence of a special basis in any algebra of length 1, $\B = \left\{1_{\A}, a_2, \ldots, a_n\right\}$ of $\A$. We will keep the notation of Definition \ref{definition3.6}

Take $a', b'$ arbitrary elements of $\A$, then there exist $\lambda, \mu \in \F$, $a, b \in \F(a_{2}, \ldots, a_{n})$ such that  $a' = \lambda 1_{\A} + a$ and $b' = \mu 1_{\A} + b$. Therefore, $$a' b' = (\lambda 1_{\A} + a)(\mu 1_{\A} + b) = \lambda \mu 1_{\A} + \lambda b + \mu a + ab.$$ In order to prove that $a'b' \in \F(1_{\A}, a', b')$, it is sufficient to prove that $ab \in \F(1_{\A}, a, b)$.

So, let $a = \lambda_{2}a_{2} + \cdots + \lambda_{n}a_{n}$ and $b = \xi_{2}a_{2} + \cdots + \xi_{n}a_{n}$ be arbitrary elements of $\F(a_{2}, \ldots, a_{n})$. Then, there are elements $\mu_1$ and $\mu_2$ $\in \F$ such that
\begin{align*}
ab &=\sum_{i,j=2}^{n}\lambda_{i}\xi_{j}a_{i}a_{j} = \sum_{\substack{i=2 \\ j\neq i}}^n \lambda_{i}\xi_{j}(\beta_{j}a_{i} -\beta_{i}a_{j}) + \mu_1 1_{\A} \\
&= \sum_{j=2}^{n}\sum_{i\neq j}(\beta_{j}a_{i})\lambda_{i}\xi_{j} - \sum_{l=2}^{n}\sum_{h \neq l} (\beta_{l}a_{h})\lambda_{l}\xi_{h} + \mu_1 1_{\A} \\
&= \sum_{j=2}^n \beta_{j}\xi_{j}\left(\sum_{i \neq j} \lambda_{i}a_{i}\right) - \sum_{l=2}^{n} \beta_{l}\lambda_{l}\left(\sum_{h \neq l} \xi_{h}a_{h}\right) + \mu_1 1_{\A} \\
&= \sum_{j=2}^{n}\beta_{j}\xi_{j}(a - \lambda_{j}a_{j}) - \sum_{l=2}^{n}\beta_{l}\lambda_{l}(b - \xi_{l}a_{l}) + \mu_2 1_{\A} \\
&= \sum_{j=2}^{n} \beta_{j}\xi_{j}a - \sum_{j=2}^{n} \beta_{j}\xi_{j}\lambda_{j}a_{j} - \sum_{l=2}^n \beta_{l}\lambda_{l}b + \sum_{l=2}^{n} \beta_{l}\lambda_{l}\xi_{l}a_{l} + \mu_2 1_{\A} \\
&= \left(\sum_{j=2}^{n} \beta_{j}\xi_{j}\right)a - \left(\sum_{j=2}^{n}\beta_{l}\lambda_{l}\right)b + \mu_2 1_{\A} \in \F(1_{\A}, a, b),
\end{align*}
which completes the proof.
\end{proof}

Notice that the previous characterization of unital algebras of length one easily allows the construction of an algorithm to decide if an arbitrary unital finite dimensional $\F$-algebra has length one.  For that we only need to know a basis and the corresponding multiplication table and make use of Lemma \ref{lemma3.1}, Proposition \ref{theorem3.5} and Theorem \ref{main}.

So let us assume, without loss of generality, that we know a basis  $\B=\left\{a_1 = 1_{\A}, a_2, ..., a_n\right\}$ of a unital algebra $\A$ over $\F$ that contains the identity $1_{\A}$.

The algorithm proceeds as follows:

\vspace{0.25cm}

ALGORITHM

\vspace{0.25cm}

{\bf Step 1.}  For any element $a_i \in \B$ check if $a_i^2 \in \F(1_{\A}, a_i)$.

\begin{itemize}
\item[(i)] If it fails for some $2 \leq i \leq n$, then $\A$ has length $> 1$;
\item[(ii)] If it is satisfied for any $2 \leq i \leq n$, then go to Step 2.
\end{itemize}

{\bf Step 2.}  Construct, from $\B$, a canonical basis $\B_c = \left\{b_1 = 1_{\A}, b_2, \ldots ,b_n\right\}$ ($b_i^2 = \mu_i 1_{\A}$)  (we know how to do it by Lemma \ref{lemma3.1}).

\vspace{0.25cm}

{\bf Step 3.}  Check if this basis $\B_c$ is special, that is, if $b_i b_j \in \F(1_{\A}, b_i, b_j)$ and $b_i b_j + b_j b_i \in \F 1_{\A}$ for any $2 \leq i \neq j \leq n$.

\begin{itemize}
\item[(I)] If $\B_c$ is special, then $l(\A) = 1$;
\item[(II)] If $\B_c$ is not special, then $l(\A) >1$.
\end{itemize}

Jordan algebras are an important class of non-associative algebras.  They were introduced by P. Jordan, J. von Neumann and E. Wigner to treat with the formalism of Quantum Mechanics \cite{JNW34}. For results about Jordan algebras we refer readers to \cite{Jac68},  \cite{MC04} and \cite{ZheSSS}.

\begin{Definition}
A Jordan algebra is an algebra  $\A$ over a field $\F$, ${\rm char} \ \F \neq 2$, satisfying the following 2 identities:
\begin{itemize}
\item[(J.1)] $xy = yx$ (commutativity);
\item[(J.2)] $x^2(yx) = (x^2y)x$ (Jordan identity).
\end{itemize}
\end{Definition}

Examples of Jordan algebras include the algebra $\A^+$, obtained from an associative algebra $\A$ changing the original associative product to the new product given by $x.y = \frac{1}{2}(xy + yx)$, and $\A = \F 1_{\A} + \mathcal{V}$, the Jordan algebra of a bilinear form, where $\mathcal{V}$ is a vector space over $\F$ with a symmetric bilinear form $\varphi \colon \mathcal{V} \times \mathcal{V} \rightarrow \F$ and the multiplication in $\A$ is given by
$$(\alpha 1_{\A} + x) (\beta 1_{\A} + y) = (\alpha \beta + \varphi(x,y))1_{\A} + \alpha y + \beta x,$$
for all $\alpha, \beta \in \F$ and $x, y \in \mathcal{V}$.

\begin{Theorem} Let $\A$ be a unital algebra with $l(\A) = 1$ and let $\B = \left\{1_{\A}, a_2, \ldots, a_n\right\}$ be a special basis of $\A$. Then the following conditions are equivalent:
\begin{itemize}
\item[$(1)$] $\A$ is a Jordan algebra;
\item[$(2)$] $\A$ is commutative;
\item[$(3)$] $\alpha_{ij} = \alpha_{ji}$, $\beta_{j} = 0$, for all $i \neq j \in \left\{2, \ldots, n\right\}$.
\end{itemize}
\end{Theorem}
\begin{proof}
It is clear that $(1) \Rightarrow (2)$.

$(2) \Rightarrow (3)$. Let $\A$ be commutative. Then, for all $i \neq j \in \left\{2, \ldots, n\right\}$ $a_i a_j = a_j a_i$, that is, $\beta_{j}a_i - \beta_{i} a_j + \alpha_{ij} 1_{\A} = \beta_{i}a_j - \beta_{j}a_i + \alpha_{ji} 1_{\A}$. So $\alpha_{ij} = \alpha_{ji}$ and $\beta_{j} = 0$.

$(3) \Rightarrow (1)$. If $\alpha_{ij} = \alpha_{ji}$ and $\beta_{j} = 0$, for every $i \neq j$, then the product of $\A$ with respect to the special basis $\B = \left\{1_{\A}, a_{2}, \ldots, a_{n}\right\}$ is given by $a_{i}^2  = \mu_i 1_{\A}, a_i a_j = a_j a_i = \alpha_{ij} 1_{\A}.$

Let us denote $\mathcal{V} = \F(a_2, \ldots, a_n)$ and $\left\langle \cdot , \cdot\right\rangle \colon \mathcal{V} \times \mathcal{V} \rightarrow \F$ the symmetric bilinear form on $\mathcal{V}$ given by:
$$\left\langle a_i, a_i\right\rangle = \mu_i, \left\langle a_i, a_j\right\rangle = \left\langle a_j, a_i\right\rangle = \alpha_{ij}, \ {\rm if} \ i \neq j.$$
So $\A = \F 1_{\A} + \mathcal{V}$ is the Jordan algebra of the bilinear form $\left\langle \cdot , \cdot\right\rangle$.
\end{proof}

\begin{Corollary}
Jordan algebras of a symmetric bilinear form are the only unital Jordan algebras of length one.
\end{Corollary}

Let us remember that an algebra is called flexible if $x(yx) =(xy)x$ for every $x,y \in \A$.  Clearly associative algebras (resp. commutative algebras) are flexible.

\begin{Theorem}\label{3.12}
Let $\A$ be an algebra of length $1$ and let $\B = \left\{1_{\A}, a_2, \ldots, a_n\right\}$ be an special basis of $\A$. Theorem $\ref{theorem3.5}$.
Then $\A$ is flexible if and only if the following conditions are satisfied:
\begin{align*}\label{identities}
\alpha_{ij} & = \alpha_{ji}, & \beta_{j} \mu_i & = \beta_{i} \alpha_{ij}, & \beta_{h} \alpha_{ij} + \beta_{i}\alpha_{hj} & = 2 \beta_{j} \alpha_{ih}, \tag{F}
\end{align*}
for arbitrary distinct i, j, h $\in \left\{2, \ldots, n\right\}$.
\end{Theorem}
\begin{proof}
If the algebra $\A$ is flexible then using the flexible identity and its linearization we get
\begin{align*}
a_i (a_j a_i) &= (a_i a_j) a_i \\
a_i (a_j a_h) + a_h (a_j a_i) &= (a_i a_j) a_h + (a_h a_j) a_i
\end{align*}
for arbitrary $i, j, h \in \left\{2, \ldots, n\right\}$.

In particular, taking $i, j, h$ distinct we have:
\begin{align*}
a_i (a_j a_i) &= (\beta_i \alpha_{ij} - \beta_j \mu_i) 1_{\A} + (\alpha_{ji} + \beta_i \beta_j) a_i - {\beta_i}^2 a_j \\
(a_i a_j) a_i &= (\beta_j \mu_i - \beta_i \alpha_{ij}) 1_{\A} + (\alpha_{ij} + \beta_i \beta_j) a_i - {\beta_i}^2 a_j \\
a_i (a_j a_h) + a_h (a_j a_i) &= (\beta_j (\alpha_{ih} + \alpha_{hi}) - \beta_i \alpha_{hj} - \beta_h \alpha_{ji}) 1_{\A} + \\
& + (\alpha_{hj} + \beta_h \beta_j) a_i + (\alpha_{ij} + \beta_i \beta_j) a_h - 2 \beta_i \beta_h a_j
\end{align*}
and
\begin{align*}
(a_i a_j) a_h + (a_h a_j) a_i &= (-\beta_j (\alpha_{ih} + \alpha_{hi}) + \beta_h \alpha_{ij} + \beta_i \alpha_{hj}) 1_{\A} + \\
& + (\alpha_{jh} + \beta_h \beta_j) a_i + (\alpha_{ji} + \beta_i \beta_j) a_h - 2 \beta_i \beta_h a_j
\end{align*}

So $\A$ flexible implies:
\begin{align*}
\beta_i \alpha_{ij} - \beta_j \mu_i &= \beta_j \mu_i - \beta_i \alpha_{ji} \\
\alpha_{ji} + \beta_i \beta_j &= \alpha_{ij} + \beta_i \beta_j \\
\beta_j (\alpha_{ih} + \alpha_{hi}) - \beta_i \alpha_{hj} - \beta_{h} \alpha_{ji} &= - \beta_j (\alpha_{ih} + \alpha_{hi}) + \beta_h \alpha_{ij} + \beta_i \alpha_{hj}
\end{align*}
that is $\alpha_{ij} = \alpha_{ji}$, $\beta_{j}\mu_i = \beta_i \alpha_{ji}$ and $\beta_h \alpha_{ij} + \beta_i \alpha_{hj} = 2 \beta_j \alpha_{ih}$ for any distinct $i, j, h \in \left\{2, 3, \ldots, n\right\}$.

Conversely, let us assume that $\A$ has an special basis $\B = \left\{1_{\A} = a_1, a_2, \ldots, a_n \right\}$ with a multiplication table:
$$a_{i}^2 = \mu_i 1_{\A}, \hspace{0.5cm} a_i a_j = \alpha_{ij}1_{\A} + \beta_j a_i - \beta_i a_j$$
satisfying the identities (F). In order to prove that $\A$ is flexible we have to prove that $a_i (a_j a_h) + a_h (a_j a_i) = (a_i a_j) a_h + (a_h a_j) a_i$ for any $i, j, h \in \left\{2, \ldots, n\right\}$.

The proof of the previous implication shows that this is true if $i, j, h$ are distinct or if $i = h \neq j$.

When $i = j = h$ the result is trivial since $a_i.{a_i}^2  = {a_i}^2.a_i = \mu_i a_i$.

If $i = j \neq h$, then $a_i (a_i a_h) + a_h.{a_i}^2 = {a_i}^2.a_h + (a_h a_i) a_i$ since
\begin{align*}
a_h.{a_i}^2 &= {a_i}^2.a_h = \mu_i a_h
\end{align*}
and
\begin{align*}
a_i (a_i a_h) &= (\beta_h \mu_i - \beta_i \alpha_{ih})1_{\A} + (\alpha_{ih} - \beta_i \beta_h)a_i + \beta_{i}^2 a_h \\
              &= (\alpha_{ih} - \beta_i \beta_h)a_i + \beta_{i}^2 a_h \\
(a_h a_i) a_i &= (\beta_i \alpha_{hi} - \beta_h \mu_i)1_{\A} + \beta_{i}^2 a_h + (\alpha_{hi} - \beta_i \beta_h) a_i \\
              &= a_i (a_i a_h)							
\end{align*}
using the identities (F).

The remaining case, $i \neq j = h$ is similar to the previous one.
\end{proof}

Let us notice that the proof of Theorem \ref{3.12} shows that we can take any special basis of $\A$ ($l(\A) = 1)$ to check if $\A$ is flexible.

\begin{Theorem}\label{theorem3.9}
Let $\A$ be an $\F$-algebra with $1_{\A}$ and $l(\A) = 1$. Let $\B = \left\{1_{\A} = a_1, a_2, \ldots, a_n\right\}$ be an special basis of $\A$. Then $\A$ is associative if and only if the following conditions are satisfied:
\begin{itemize}
\item[{\rm (A1)}] $\mu_i = \beta_{i}^2$, for every $i \in \left\{2, \ldots, n\right\}$;
\item[{\rm (A2)}] $\alpha_{ij} = \beta_{j} \beta_{i} = \alpha_{ji}$, for all $i \neq j \in \left\{2, \ldots, n\right\}$.
\end{itemize}
\end{Theorem}
\begin{proof}
Let us assume that $\A$ is associative. By simplicity, we will consider first the case ${\rm dim} \ \A = 3$. Then $\A$ is associative if and only if $a_i (a_j a_k) = (a_i a_j) a_k$ for every $\left\{i, j, k\right\} \subseteq \left\{2, 3\right\}$. By symmetry we only need to consider the identities: $a_2(a_2a_3) = a_2^2 a_3$ and $a_2(a_3a_2) = (a_2a_3)a_2$. But  $a_{2}^2 a_3 = \mu_2 a_3$ and
\begin{align*}
a_2 (a_2 a_3) &= a_2 (\alpha_{23} 1_{\A} + \beta_{3} a_2 - \beta_{2}a_3) \\
&= \alpha_{23}a_2 + \beta_{3}\mu_2 1_{\A} - \beta_{2}(\alpha_{23}1_{\A} + \beta_{3}a_2-\beta_{2}a_3) \\
&= (\beta_{3}\mu_2 - \beta_{2}\alpha_{23})1_{\A} + (\alpha_{23} - \beta_{2}\beta_{3})a_2 + \beta_{2}^{2}a_{3}.
\end{align*}

So associativity of $\A$ implies $0 = \beta_{3}\mu_2 - \beta_{2}\alpha_{23}$, $\alpha_{23} = \beta_{2} \beta_{3}$, $\mu_2 = \beta_{2}^{2}$, and the first identity follows from the other two: $\beta_{3}\mu_2 = \beta_{3} \beta_{2}^2 = \alpha_{23}\beta_{2}$. Consequently associativity of $\A$ implies (by symmetry) $\alpha_{23} = \alpha_{32} = \beta_{2} \beta_{3}$ and $\mu_2 = \beta_{2}^2$, $\mu_{3} = \beta_{3}^2$.

Conversely, if the previous identities are satisfied then $a_2 (a_2 a_3) = a_2^2 a_3$ and $a_3 (a_3 a_2) = a_3^{2}a_2$. We only need to check that $a_2 (a_3 a_2) = (a_2 a_3) a_2$ and $a_3 (a_2 a_3) = (a_3 a_2)a_3$. This follows immediately from previous result, since our conditions guarantee that the algebra $\A$ is flexible.

So conditions (A1) + (A2) imply associativity of $\A$ when ${\rm dim} \ \A = 3$.

Now let us consider the general case. If $\A$ is associative, ${\rm dim} \ \A = n$, $l(\A) = 1$ and $\left\{1_{\A}, a_2, \ldots, a_n\right\}$ is an special basis of $\A$, then for all $i \neq j \in \left\{2, \ldots, n\right\}$ the algebra $\F(1_{\A}, a_i, a_j)$ is 3-dimensional and associative. As we have just seen, this implies that $\alpha_{ij} = \alpha_{ji} = \beta_{j} \beta_{i}$ and $\mu_i = \beta_{i}^2$.

Conversely, let us assume that $\alpha_{ij} = \alpha_{ji} = \beta_{j} \beta_{i}$, for all $i \neq j \in \left\{2, \ldots, n\right\}$, and $\mu_i = \beta_{i}^2$ for any $i \in \left\{2, \ldots, n\right\}$. So for any $i \neq j$ in $\left\{2, \ldots, n\right\}$ we already know, using the 3-dimensional case studied above, that $a_i(a_k a_j) = (a_i a_k) a_j$ and $a_j (a_k a_i) = (a_j a_k) a_i$, when $k=i$ or $j$. Now suppose that $i, j, k$ are 3 distinct elements in $\left\{2, \ldots, n\right\}$. We need to prove that $(a_i a_j) a_k = a_i (a_j a_k)$ in order to get associativity.

\begin{align*}
(a_i a_j) a_k& = (\alpha_{ij}1_{\A} + \beta_{j} a_i - \beta_{i} a_j) a_k \\
&=\alpha_{ij} a_k + \beta_{j} (\alpha_{ik} 1_{\A} + \beta_{k} a_i - \beta_{i}a_k) - \beta_{i} (\alpha_{jk}1_{\A} + \beta_{k}a_j - \beta_{j}a_k) \\
&=(\beta_{j} \alpha_{ik} - \beta_{i} \alpha_{jk}) 1_{\A} + \beta_{j} \beta_{k} a_i - \beta_{i} \beta_{k} a_j + (\alpha_{ij} - \beta_{i}\beta_{j} + \beta_{i}\beta_{j})a_k \\
&= (\beta_{j} \beta_{i} \beta_{k} - \beta_{i} \beta_{j} \beta_{k}) 1_{\A} + \alpha_{jk} a_i - \alpha_{ik}a_j + \alpha_{ij}a_k \\
&= \alpha_{jk} a_i - \alpha_{ik}a_j + \alpha_{ij}a_k.
\end{align*}

In the same way
$$a_i (a_j a_k) = \alpha_{jk} a_i - \alpha_{ik}a_j + \alpha_{ij}a_k.$$
So the algebra $\A$ is associative.
\end{proof}

\begin{Remark}
Theorem \ref{theorem3.9} allows to get examples of unital algebras of length 1 that are not associative. For instance, the algebra $\F(1_{\A}, a, b)$ with the product given by $a^2 = 0$, $b^2 = b$, $ab = 2(1_{\A}) + a + b$, $ba = -a - b$ is non-associative and has length equal to 1.
\end{Remark}

\section{Characteristic 2}

In this section we will address the case of characteristic 2.  So $\F$ will denote an extension of the field of two elements $\F_{2}$. We will study first, in a separate way, the case of dimension 3 that behaves in a different way.

In what follows,  given $a, b$ arbitrary elements of $\A$, we denote by $a \equiv b \ ({\rm mod} \ 1_{\A})$ (or simply by $a \equiv b$) if and only if $a - b \in \F 1_{\A}$.

\begin{Lemma}\label{lemma4.1}
Let $\A$ be a unital $3$-dimensional $\F$-algebra and $\left\{b_1 = 1_{\A}, b_2, b_3\right\}$ an arbitrary basis of $\A$. Then $l(\A) = 1$ implies that the multiplication table of $\A$ satisfies the following conditions:
\begin{align*}
b_{2}^2  &\equiv \delta_{2} b_{2}, & b_{3}^2 &\equiv \delta_{3} b_3,   \\
b_{2} b_{3} &\equiv \beta_2 b_2 + \beta_{3}^{\ast} b_3,  & b_3 b_2 \equiv &\beta_{3} b_{3} + \beta_{2}^{\ast} b_{2}, \tag{C2}
\end{align*}
where $\delta_2, \beta_2, \beta_2^{\ast}, \delta_3, \beta_3, \beta_3^{\ast} \in \F$, and $\beta_2 + \beta_2^{\ast} + \delta_2 = \beta_3 + \beta_3^{\ast} + \delta_3.$
\end{Lemma}
\begin{proof}
Suppose that  $l(\A) = 1$.  By Proposition \ref{proposition0}, there exist $ \beta_{2}, \beta_{3}^{*}, \beta_{3}, \beta_{2}^{*} \in \F$ such that  $b_{2}b_{3} \equiv  \beta_{2}b_{2} + \beta_{3}^{*}b_{3}$ and $b_{3}b_{2} \equiv  \beta_{3}b_{3} + \beta_{2}^{*}b_{2}$.

By Corollary \ref{lemma0}, there exist $\delta_{2}, \delta_{3} \in \F$ such that $b_{2}^{2} \equiv \delta_{2}b_{2}$, $b_{3}^{2} \equiv \delta_{3}b_{3}$.

Corollary \ref{lemma0} also gives that $$(b_{2} + b_{3})^{2} = b_{2}^{2} + b_{2}b_{3} + b_{3}b_{2} + b_{3}^{2} \equiv \delta_{2}b_{2}  + \beta_{2}b_{2} + \beta_{3}^{*}b_{3} + \beta_{3}b_{3} + \beta_{2}^{*}b_{2} + \delta_{3}b_{3}$$ $\in \F_{2}(1_{\A}, b_{2} + b_{3}),$
which implies $\beta_{3}^{*} + \beta_{3} + \delta_{3} = \beta_{2} + \beta_{2}^{*} + \delta_{2}.$
\end{proof}

\begin{Theorem}\label{theorem4.2}
Let $\A$ be a unital $\F_{2}$-algebra of dimension $3$.  Then $l(\A) = 1$ if and only if there is a basis $\B = \left\{a_{1} = 1_{\A}, a_{2}, a_{3}\right\}$ whose multiplication table satisfies one of the following 3 conditions:
\begin{itemize}
\item[{\rm 1.}] $a_2^2 \equiv 0$, $a_3^2 \equiv 0$, $a_2 a_3 \equiv 0$, $a_3 a_2 \equiv 0$;
\item[{\rm 2.}] $a_2^2 \equiv a_2$, $a_3^2 \equiv a_3$, $a_2 a_3 \equiv 0$, $a_3 a_2 \equiv 0$;
\item[{\rm 3.}] $a_2^2 \equiv 0$, $a_3^2 \equiv a_3$, $a_2 a_3 \equiv 0$, $a_3 a_2 \equiv a_3$.
\end{itemize}
\end{Theorem}

\begin{proof}
Let $\left\{1_{\A} = b_1, b_2, b_3\right\}$ be an arbitrary basis of $\A$. If $l(\A) = 1$, then $b^2 \in \F(1_{\A}, b)$ for any $b \in \A$. So $b^2 \equiv 0$ or $b^2 \equiv b$. (Notice that $1_{\A}^2 \equiv 0$ and $1_{\A}^2 \equiv 1_{\A}$ and it is the only nonzero element that satisfies both conditions). We will distinguish 3 cases.

{\bf Case 1}: For every $a \in \A$, $a^2 \equiv 0$.

Using Lemma \ref{lemma4.1} we know that there are elements $\beta_2$, $\beta_2^{\ast}$, $\beta_3$, $\beta_3^{\ast}$ $\in \F_2$ (notice that $\delta_2 = \delta_3 = 0)$ such that $b_2 b_3 \equiv \beta_2 b_2 + \beta_2^{\ast} b_3$, $b_3 b_2 \equiv \beta_3 b_3 + \beta_2^{\ast} b_2$ and $\beta_2 + \beta_2^{\ast} = \beta_3 + \beta_3^{\ast}$. Then $b_2 b_3 + b_3 b_2 \equiv (\beta_3 + \beta_3^{\ast})b_3 + (\beta_2 + \beta_2^{\ast})b_2 = \beta(b_2 + b_3)$, where $\beta = \beta_2 + \beta_2^{\ast} = \beta_3 + \beta_3^{\ast}$. But $(b_2 + b_3)^2 \equiv 0$ by our assumption, so $\beta = 0$. Then $\beta_2 = \beta_2^\ast$, $\beta_3  = \beta_3^{\ast}$ and so $b_2 b_3 \equiv \beta_2 b_2 + \beta_3 b_3  \equiv b_3 b_2$ and $b_2 b_3 + b_3 b_2  \equiv 0$.

Taking the basis $\B_1 = \left\{a_1 = 1_{\A}, a_2 = \beta_3 1_{\A} + b_2, a_3  = \beta_2 1_{\A} + b_3\right\}$ the multiplication table satisfies $a_2^2 \equiv 0$, $a_3^2 \equiv 0$, $a_2 a_3 \equiv \beta_3 b_3 + \beta_2 b_2 + b_2 b_3  = 0$ and $a_3 a_2 \equiv 0$.

{\bf Case 2}. For every $a \not\equiv 0$, $a^2 \equiv a$.

Again Lemma \ref{lemma4.1} implies that there are elements $\beta_2, \beta_2^{\ast}, \beta_3, \beta_3^{\ast} \in \F_2$ (notice that $\delta_2 = \delta_3 = 1)$ such that $\beta_2 + \beta_2^{\ast} = \beta_3 + \beta_3^{\ast}$, $b_2^2 \equiv b_2$, $b_3^2 \equiv b_3$, $b_2 b_3 \equiv \beta_2 b_2 + \beta_3^{\ast} b_3,$ $b_3 b_2 \equiv \beta_3 b_3 + \beta_2^{\ast} b_2$.

So $b_2 b_3 + b_3 b_2 \equiv (\beta_2 + \beta_2^{\ast}) b_2 + (\beta_3 + \beta_3^\ast) b_3 = \beta (b_2 + b_3)$, where $\beta = \beta_2 + \beta_2^{\ast} = \beta_3 + \beta_3^{\ast}$. But our assumption implies that $(b_2 + b_3)^2 \equiv b_2 + b_3$ and $(b_2 + b_3)^2 = b_2^2 + b_2 b_3 + b_3 b_2 + b_3^2 \equiv b_2 + b_3 + \beta (b_2 + b_3)$.

So $\beta = 0$ and $\beta_2  = \beta_2^{\ast}$, $\beta_3 = \beta_3^{\ast}$.

That is, the multiplication table is $b_2 b_3 \equiv \beta_2 b_2 + \beta_3 b_3 \equiv b_3 b_2$.

Take $a_2 = \beta_3 1_{\A} + b_2,$ $a_3 = \beta_2 1_{\A} + b_3$ as before. Then $\B_2 = \left\{a_1 = 1_{\A}, a_2, a_3\right\}$ is a basis of $\A$, $a_2^2 \equiv \beta_{3} b_2 + \beta_3 b_2 + b_2^2 \equiv b_2 \equiv a_2$, $a_3^2 \equiv a_3$, $a_2 a_3 = a_3 a_2 \equiv \beta_3 b_3 + \beta_2 b_2 + b_2 b_3 \equiv 0$.

{\bf Case 3.} There are elements $b_2 \not\equiv 0$, $b_3 \not\equiv 0$ such that $b_2^2 \equiv 0$, $b_3^2 \equiv b_3$. Then $\left\{1_{\A}, b_2, b_3\right\}$ is a basis of $\A$. Again, using Lemma \ref{lemma4.1} we know that there are elements $\beta_2, \beta_2^{\ast}, \beta_3, \beta_3^{\ast} \in \F_{2}$ $(\delta_2 = 0 , \delta_3 = 1)$ such that $b_2 b_3 \equiv \beta_2 b_2 +  \beta_3^{\ast} b_3$, $b_3 b_2 = \beta_3 b_3 + \beta_2^{\ast} b_2$, $\beta_2 + \beta_2^{\ast} = \beta_3 + \beta_3^{\ast} + 1$. If $\beta_2 + \beta_2^{\ast}  = 0$, then $\beta_2 = \beta_2^{\ast}$, $\beta_3 = \beta_3^{\ast} + 1$. So $b_2 b_3 + b_3 b_2  = \beta_2 b_2 + (\beta_3 + 1) b_3 + \beta_3 b_3 + \beta_2 b_2 = b_3$.

Taking $a_2 = (\beta_3 + 1)1_{\A} + b_2$, $a_3 = \beta_2 1_{\A} + b_3$, $\B_3 = \left\{a_1 = 1_{\A}, a_2, a_3 \right\}$ is a basis of $\A$, $a_2^2 \equiv b_2^2 \equiv 0$, $a_3^2 \equiv b_3^2 \equiv b_3 \equiv a_3$,
$$a_2 a_3 = (\beta_3 + 1)b_3 + \beta_2 b_2 + \beta_2 b_2 + (\beta_3 + 1) b_3 \equiv 0,$$
$$a_3 a_2 \equiv \beta_2 b_2 + (\beta_3 + 1) b_3 + \beta_3 b_3 + \beta_2 b_2 = b_3 \equiv a_3.$$

Clearly algebras satisfying two different conditions are not isomorphic.

For the converse notice that $\F_{2}$-algebras satisfying one of the 3 conditions above satisfy also ${\rm (C2)}$.

So it suffices to prove that if an algebra $\A$ over $\F_2$ has a basis that satisfies condition ${\rm (C2)}$ then the algebra has length 1.

Let $\left\{a_1 = 1_{\A}, a_2, a_3\right\}$ be a basis of $\A$ with the indicated multiplication table. In order to prove that $l(\A) = 1$ we only need to prove that for arbitrary elements  $a = \lambda_{2}a_{2}+\lambda_{3}a_{3}$, $b= \mu_{2}a_{2} + \mu_{3}a_{3}$ in $\F_{2}(a_{2}, a_{3})$, the product $ab \in \F_{2}(1_{\A}, a, b)$.

The result is clear if at least one of the elements ($a$ or $b$) is equal to $0$.  So we can assume that $a,b \in \{a_2, a_3, a_2 + a_3\}$.

Again the result is clear if $a = b$ since we have proved in the first implication that  $(a_{2} + a_{3})^{2} \in \F_{2}(1_{\A}, a_{2} + a_{3})$ if and only if
 $\beta_{3}^{*} + \beta_{3} + \delta_{3} = \beta_{2} + \beta_{2}^{*} + \delta_{2}.$

The case $a = a_2$, $b = a_3$ is given directly by the conditions and the case $a = a_2$, $b = a_2 + a_3$ (similarly $a = a_3$, $b = a_2 + a_3$) is clear, since   $$ ab = a_2(a_2 + a_3) \equiv \delta_2 a_2 + \beta_{2}a_2 + \beta_{3}^{*} a_3 = ( \delta_{2} + \beta_{2} + \beta_{3}^{*})a_2 + \beta_{3}^{*}(a_2 + a_3),$$
that is,  $  a_2(a_2 + a_3) \in \F_{2}(1_{\A}, a_2, a_2 + a_3)$.
\end {proof}

\begin{Remark}
Notice that the algebra $\F_2 \oplus \F_2 \oplus \F_2$ that appears mentioned in Theorem \ref{markova} as one associative algebra of length 1 appears here as one of the algebras of type 2.
\end{Remark}

\begin{Theorem}\label{theorem4.3}
Let $\F$ be a proper extension of $\F_2$ and  $\A$ a unital $\F$-algebra of dimension $3$. Then $l(\A) = 1$ if and only if there is a basis $\B = \left\{a_{1} = 1_{\A}, a_{2}, a_{3}\right\}$ of $\A$ whose multiplication table satisfies the following conditions:
\begin{itemize}
\item[{\rm (i)}] $a_{2}^2  \equiv \delta_{2} a_{2}$,  $a_{3}^2 \equiv \delta_{3} a_3$, $\delta_2, \delta_3 \in \F_2$;
\item[{\rm (ii)}] $a_{2} a_{3} \equiv \beta_2 a_2 + \beta_{3}^{\ast} a_3$,   $a_3 a_2 \equiv \beta_{3} a_{3} + \beta_{2}^{\ast} a_{2}$, $\beta_2$, $\beta_2^{\ast}$, $\beta_3$, $\beta_3^{\ast} \in \F$ and $\beta_2 + \beta_2^{\ast} + \delta_3 = 0 = \beta_3 + \beta_3^{\ast} + \delta_2$.
\end{itemize}
\end{Theorem}
\begin{proof}
Assume that $l(\A) = 1$. Take $\left\{b_{1} = 1_{\A}, b_{2}, b_{3}\right\}$  a basis of $\A$. By Corollary~\ref{lemma0} we know that $b_{i}^{2} \equiv \gamma_{i}b_{i}$, $i = 2, 3$, for some $\gamma_{i} \in \F$.  If $\gamma_{i} \neq 0$, taking $a_i = \gamma_{i}^{-1}b_{i}$, we get a basis $\B = \left\{1_{\A}, a_{2}, a_{3}\right\}$ of $\A$ satisfying that $a_{2}^{2} \equiv \delta_{2}a_{2}$ and $a_{3}^{2} \equiv \delta_{3}a_{3}$, $\delta_i =0$ or $1$, that is, $\delta_{2}, \delta_{3} \in \F_{2}$.

Furthermore,  $a_{2}a_{3} \equiv \beta_{2}a_{2} + \beta_{3}^{*}a_{3}$,  and $a_{3}a_{2} \equiv \beta_{3}a_{3} + \beta_{2}^{*}a_{2}$ for some $\beta_{2}, \beta_{3}, \beta_{2}^{*}, \beta_{3}^{*} \in \F$, by Proposition \ref{proposition0}.

Corollary \ref{lemma0} implies that, for all $\lambda \in \F$,
$$(a_{2} + \lambda a_{3})^{2} \equiv \delta_{2}a_{2} + \lambda(\beta_{2}a_{2}+\beta_{3}^{*}a_{3}+ \beta_{3}a_{3}+\beta_{2}^{*}a_{2}) + \lambda^{2} \delta_{3}a_{3} \in \F(1_{\A}, a_{2} + \lambda a_{3}).$$

This implies that  $\lambda(\delta_{2} +\lambda\beta_{2} +  \lambda \beta_{2}^{*}) = \lambda \beta_{3}^{*} + \lambda \beta_{3} + \lambda^{2} \delta_{3} $, that is,  $\lambda (\beta_{3}^{*} + \beta_{3} + \delta_{2}) = \lambda^{2}(\beta_{2}+\beta_{2}^{*}+\delta_{3}).$

Since $| \F| \geq 4$, we can conclude that $\beta_{3}^{*} + \beta_{3} + \delta_{2} = 0 = \beta_{2}+\beta_{2}^{*}+\delta_{3}$.

Conversely, let  us assume the existence of a basis whose multiplication table satisfies the conditions (i) and (ii), as described in the theorem.  By Proposition \ref{proposition0} we only need to prove that for arbitrary elements $a, b$  of $\A$,  $ab \in \F(1_{\A}, a, b)$.  Since this claim is obviously true when $\{1_{\A}, a, b\}$ is a basis, we only have to consider the case in which the 3 elements are linearly dependent, what reduces to prove that, for an arbitrary element $a \in \A$ the element $a^2 \in \F(1_{\A}, a)$.  Since every element  $a \in \A$ can be expressed as $a = \mu 1_{\A} + a'$, where $a' \in \F(a_{2},a_{3})$, and  $a^2 = \mu^2 1_{\A} + a'^2$, the claim is true for any $a \in \F(1_{\A}, a_2, a_3)$ if and only if it is true for any $a' \in \F(a_2, a_3)$. So, let us assume that $a = \lambda_{2}a_{2}+\lambda_{3}a_{3}$. First of all notice that  $a_{2}a_{3} + a_{3}a_{2} \equiv \delta_{3}a_{2} + \delta_{2}a_{3}$. Then
$$a^2 = \lambda_{2}^2 a_{2}^2 + \lambda_{2}\lambda_{3}(a_{2} a_{3} + a_{3} a_{2})  + \lambda_{3}^2 a_{3}^2 \equiv \lambda_{2}^2 \delta_{2}a_{2} + \lambda_{2}\lambda_{3}(\delta_{2} a_{3} + \delta_{3} a_{2}) +
\lambda_{3}^2 \delta_3 a_{3}.$$

If $\delta_2 = \delta_3 = \delta$, we get $a^2 \equiv ( \lambda_{2}^2 + \lambda_{2}\lambda_{3}) \delta a_{2} + (\lambda_{2}\lambda_{3} +
\lambda_{3}^2) \delta a_{3},$ that is, $$a^2 \equiv (\lambda_{2} + \lambda_{3}) \delta (\lambda_{2} a_{2} + \lambda_{3}a_{3}),$$ what  proves the result in this case.

When  $\delta_2 \neq \delta_3$, we can assume, without loss of generality,  that $\delta_2 = 1$ and $\delta_3 = 0$.  So
$a^2  \equiv \lambda_{2}^2 a_{2} + \lambda_{2}\lambda_{3} a_{3} = \lambda_{2} a,$
what finishes the proof of the theorem.
\end{proof}

\begin{Corollary}\label{corollary4.5}
Let $\F$ be a proper extension of $\F_2$ and $\A$ a unital $\F$-algebra of dimension {\rm 3}. Then $l(\A) = 1$ if and only if there is a basis  $\B^{\ast} = \left\{a_1^{\ast} = 1_{\A}, a_2^{\ast}, a_3^{\ast} \right\}$ whose multiplication table satisfies one of the following 3 conditions:
\begin{itemize}
\item[{\rm 1.}] ${a_2^\ast}^2 \equiv 0, {a_3^\ast}^2 \equiv 0, a_2^\ast  a_3^\ast \equiv 0$, $a_3^\ast a_2^\ast \equiv 0$;
\item[{\rm 2.}] ${a_2^\ast}^2 \equiv a_2^\ast$, ${a_{3}^\ast}^2 \equiv a_3^\ast$, $a_2^\ast a_3^\ast \equiv a_3^\ast, a_3^\ast a_2^\ast \equiv a_2^\ast$;
\item[{\rm 3.}] ${a_2^\ast}^2 \equiv 0, {a_3^\ast}^2 \equiv a_3^\ast$, $a_2^\ast a_3^\ast \equiv 0$, $a_3^\ast a_2^\ast \equiv a_3^\ast$.
\end{itemize}
\end{Corollary}
\begin{proof}
By Theorem \ref{theorem4.3} we know that if $\A$ is an algebra of length 1 then there is a basis $\B = \left\{a_1 = 1_{\A}, a_2, a_3\right\}$ with a multiplication table given by
\begin{align*}
a_{2}^2  &\equiv \delta_{2} a_{2}, & a_{3}^2 &\equiv \delta_{3} a_3   \\
a_{2} a_{3} &\equiv \beta_2 a_2 + \beta_{3}^{\ast} a_3,  & a_3 a_2 \equiv &\beta_{3} a_{3} + \beta_{2}^{\ast} a_{2},
\end{align*}
where $\delta_2, \beta_2, \beta_2^{\ast}, \delta_3, \beta_3, \beta_3^{\ast} \in \F$, and $\beta_2 + \beta_2^{\ast} + \delta_3 = 0 = \beta_3 + \beta_3^{\ast} + \delta_2.$

We will distinguish 3 cases, when both $\delta_2$ and $\delta_3$ are either 0 or 1 and when one of them is 0 and the other is 1.

{\bf Case 1.} If $\delta_2 = \delta_3 = 0$, $b_2^2 \equiv 0$, $b_3^2 \equiv 0$, $\beta_2 + \beta_2^{\ast} = 0 = \beta_3 + \beta_3^{\ast}$, then $\beta_2 = \beta_2^{\ast}$, $\beta_3 = \beta_3^{\ast}$, then $a_2 a_3 \equiv \beta_2 a_2 + \beta_3 a_3 \equiv a_3 a_2$, then $a_2 a_3 + a_3 a_2 \equiv 0$.

Taking $a_2^\ast  = \beta_3 1_{\A} + a_2$, $a_3^\ast = \beta_2 1_{\A} + a_3$, the basis $\B_1^\ast = \left\{a_{1}^\ast = 1_{\A}, a_2^\ast, a_3^\ast \right\}$ satisfies condition 1.

{\bf Case 2.} If $\delta_2 = \delta_3 = 1$, $a_2^2 \equiv a_2$, $a_3^2 \equiv a_3$, $\beta_2^{\ast}  = \beta_2 + 1$, $\beta_3^{\ast} = \beta_{3} + 1$, then $a_2 a_3  \equiv \beta_2 a_2 + (\beta_3 + 1) a_3$, $a_3 a_2 \equiv (\beta_2 + 1) a_2 + \beta_3 a_3$.

Consider the basis $\B_2^\ast = \left\{a_1^\ast = 1_{\A}, a_2^\ast = a_2 + \beta_3 1_{\A}, a_3^\ast  = a_3 + \beta_2 1_{\A}\right\}$. It satisfies condition 2.

{\bf Case 3.} If $\delta_2 = 0, \delta_3 = 1$, then $\beta_2^{\ast} = \beta_2 + 1,$ $\beta_3^{\ast} = \beta_3$. So $a_2^2 \equiv 0$, $a_3^2 \equiv 1$, $a_2 a_3 = \beta_2 a_2 + \beta_3 a_3$, $a_3 a_2  = \beta_2 a_2 + (\beta_3 + 1) a_3$. Take $a_2^\ast = a_2 + \beta_3 1_{\A}$ and $a_3^\ast = a_3 + \beta_2 1_{\A}$. Then the basis $\B_3^\ast = \left\{a_1^\ast = 1_{\A}, a_2^\ast, a_3^\ast \right\}$ satisfies condition 3.
\end{proof}

Finally we will address the case of dimension greater than 3.

\begin{Theorem}\label{4.6}
Let $\F$ be a field of characteristic {\rm 2} and $\A$ an $\F$-algebra with ${\rm dim} \ \A \geq 4$. Then $l(\A) = 1$ if and only if there is a basis $\B^\ast = \left\{a_{1}^\ast = 1_{\A}, a_{2}^\ast, \ldots, a_{n}^\ast\right\}$ of $\A$ whose multiplication table satisfies one of the following two conditions:
\begin{itemize}
\item[{\rm (I)}] ${a_{i}^\ast}^{2} \equiv 0$, for $i = 2, \ldots, n$, $a_{i}^\ast a_{j}^\ast \equiv a_{j}^\ast a_{i}^\ast \equiv \beta_{ij}a_{i}^\ast + \beta_{ji}a_{j}^\ast$,
\item[{\rm (II)}] ${a_{i}^\ast}^{2} \equiv a_{i}^\ast$, for $i = 2, \ldots, n$, $a_{i}^\ast a_{j}^\ast \equiv \beta_{ij}a_{i}^\ast + (1+\beta_{ji})a_{j}^\ast$.
\end{itemize}
Notice that in the last case $a_{i}^\ast a_{j}^\ast + a_{j}^\ast a_{i}^\ast \equiv a_{i}^\ast + a_{j}^\ast$.
\end{Theorem}
\begin{proof}
Assume $l(\A) = 1$ and consider $\left\{1_{\A}, b_{2}, \ldots, b_{n}\right\}$ a basis of $\A$. Then $b_{i}^{2} \equiv \delta_{i}b_{i}$, for some $\delta_{i} \in \F$, $2 \leq i \leq n$, by Corollary \ref{lemma0}.

 Arguing as in Theorem \ref{theorem4.3}, we can assume that $\delta_{i} \in \left\{0, 1\right\}$, for all $2 \leq i \leq n$.

Furthermore, by Proposition \ref{proposition0},  $b_{i}b_{j} \in \F(1_{\A}, b_{i}, b_{j})$, for all $2 \leq i \neq j \leq n$.  So $b_{i}b_{j} \equiv \beta_{ij}b_{i} + \beta_{ji}^{*}b_{j}$, for some $\beta_{ij}, \beta_{ji}^{*} \in \F$. Using again Proposition \ref{proposition0}, we know that for all distinct $i, j, k \in \left\{2, \ldots, n\right\}$, the following conditions hold:

\begin{itemize}
\item[{\rm (i)}] $b_{i}(b_{j}+b_{k}) \equiv \beta_{ij}b_{i} + \beta_{ji}^{*} b_{j} + \beta_{ik}b_{i} + \beta_{ki}^{*}b_{k} \in \F(1_{\A}, b_{i}, b_{j} + b_{k})$, which implies $ \beta_{ji}^{*} = \beta_{ki}^{*}$. Let us define $\beta_i^{*} = \beta_{ji}^{*}$ for any $j \neq i \in \left\{2, \ldots, n\right\}$;
\item[{\rm (ii)}] $(b_{j} + b_{k})b_{i} \equiv \beta_{ji}b_{j} + \beta_{ij}^{*} b_{i} + \beta_{ki}b_{k} + \beta_{ik}^{*}b_{i} \in \F(1_{\A}, b_{j}+b_{k},b_{i})$, which implies $\beta_{ji} = \beta_{ki}$. Let us define $\beta_{i} = \beta_{ji}$ for any $j \neq i \in \left\{2, \ldots, n\right\}$;
\item[{\rm (iii)}] $(b_{i} + b_{j})(b_{i} + b_{k}) \equiv \delta_{i}b_{i} + \beta_{k}b_{i} + \beta_{i}^{*}b_{k} + \beta_{i}b_{j} + \beta_{j}^{*}b_{i} + \beta_{k}b_{j} + \beta_{j}^{*}b_{k} \in \F(1_{\A}, b_{i} + b_{j}, b_{i} + b_{k})$ which implies $\delta_{i} + \beta_{k} + \beta_{j}^{*} = \beta_{i} + \beta_{k} + \beta_{i}^{*} + \beta_{j}^{*}$. Then, by (i) and (ii), we can conclude that $\beta_{i} + \beta_{i}^{*} = \delta_{i}$.
\end{itemize}

Therefore, $b_{i}b_{j} \equiv \beta_{j} b_{i} + (\beta_{i} + \delta_{i})b_{j}$ and $b_{j}b_{i} \equiv \beta_{i} b_{j} + (\beta_{j} + \delta_{j})b_{i}$, for all $i, j \in \left\{2, \ldots, n\right\}$ such that $i \neq j$.

{\bf Claim.} We can assume that $\delta_{2} = \cdots = \delta_{n}$.

Indeed, if either all scalars are 0 or all scalars are 1, there is nothing to prove. Suppose that $r-1$ scalars are 0 and $n-r$ are 1. Without loss of generality (reordering elements), we can assume that $\delta_{2} = \cdots = \delta_{r} = 0$ and $\delta_{r+1} = \cdots = \delta_{n} = 1$.

 Observe that
$$(b_{s} + b_{r+1})^{2}  \equiv \delta_{s}b_{s} + \beta_{r+1}b_{s} + \beta_{s}^{*} b_{r+1} + \beta_{s} b_{r+1} + \beta_{r+1}^{*}b_{s} + \delta_{r+1}b_{r+1}$$
$$= (\beta_{r+1} + \beta_{r+1}^{*})b_{s} + (\beta_{s} + \beta_{s}^{*} + 1)b_{r+1} = \delta_{r+1}b_{s} + (\delta_{s}+1)b_{r+1} = b_{s} + b_{r+1},$$
for all $2 \leq s \leq r$.

Thus, $\B = \left\{1_{\A}, a_{2} = b_{2} + b_{r+1}, \ldots, a_{r} = b_{r} + b_{r+1}, a_{r+1} = b_{r+1}, \ldots, a_{n} = b_{n} \right\}$ is a basis of $\A$ such that $a_{i}^{2} \equiv a_{i}$.

 In any case, we can consider a basis $\B^{\ast} = \left\{a_{1}^\ast = 1_{\A},a_2^\ast, \ldots, a_{n}^\ast\right\}$ of $\A$ such that ${a_{i}^\ast}^{2} \equiv \delta a_{i}^\ast$, for all $2 \leq i \leq n$, where $\delta \in \left\{0, 1\right\}$.  Clearly, $\delta = 0$ gives statement (1), and  $\delta = 1$ gives statement (2).

Conversely, suppose that there is a basis $\B^{\ast} = \left\{a_{1}^\ast = 1_{\A}, a_{2}^\ast, \ldots, a_{n}^\ast\right\}$ that satisfies either (I) or (II).  Let $a = \lambda_1 1_{\A} + \cdots + \lambda_n {a_n}^\ast$  and $b = \gamma_1 1_{\A} + \cdots + \gamma_n a_n^\ast$ be arbitrary elements of $\A$.  Then

{\bf Case 1:} If ${a_{i}^\ast}^{2} \equiv 0$, for $i = 2, \ldots, n$, $a_{i}^\ast a_{j}^\ast \equiv a_{j}^\ast a_{i}^\ast \equiv \beta_{j}a_{i}^\ast + \beta_{i}a_{j}^\ast$,
for all $i \neq j = 2, \ldots, n$, then
$$ab \equiv (\gamma_{2}\beta_{2} + \cdots + \gamma_{n}\beta_{n})a + (\lambda_{2}\beta_{2} + \cdots + \lambda_{n}\beta_{n})b \in \F(1_{\A}, a, b).$$

{\bf Case 2:} If ${a_{i}^\ast}^{2} \equiv a_{i}^\ast$, for $i = 2, \ldots, n$, $a_{i}^\ast a_{j}^\ast \equiv \beta_{j}a_{i} + (1+\beta_{i})a_{j}$, for all $i \neq j = 2, \ldots, n$, then
$$ab \equiv (\gamma_{2}\beta_{2} +  \cdots + \gamma_{n}\beta_{n})a + (\lambda_{2} + \cdots + \lambda_{n} + \lambda_{2}\beta_{2} + \cdots + \lambda_{n}\beta_{n})b \in \F(1_{\A}, a, b).$$

In both cases, we can conclude that $l(\A) = 1$, by  Proposition \ref{proposition0}, which completes our proof.
\end{proof}

\begin{Remark}
If $a = \lambda_1 1_{\A} + \lambda_2 a_2 + \cdots + \lambda_n a_n$ is an arbitrary element of an algebra $\A$ satisfying Theorem \ref{4.6} (I)  then $a^2 \equiv 0$. Similarly, if $\A$ satisfies Theorem \ref{4.6} (II) then $a^2 \equiv (\lambda_2 + \cdots + \lambda_n)a.$ So Theorem \ref{4.6} is a kind of classification theorem in the line of Theorem \ref{markova}. The same applies to Theorem \ref{theorem4.2} and Corollary \ref{corollary4.5}.
\end{Remark}

\section*{Conclusion}
We have given a complete characterization of algebras of length 1, without additional conditions neither on the algebra nor on the characteristic of the field.

In case of characteristic 2 we have obtained a kind of classification similar to the one obtained in the associative case (see Theorem \ref{markova}).

In case of ${\rm char}(\F) \neq 2$, our result characterizes algebras of length one through the existence of a particular basis (an {\it special basis}).  Of course there are classification problems that remain opened  and that probably only could be addressed under additional conditions on the algebra.

However our result allows the construction of an easy algorithm to decide if a given algebra, with some known basis and the corresponding multiplication table, has length one.

\section*{Acknowledgements}
Authors acknowledge the referee for many valuable comments that have helped to improve the paper. 

The investigations of the first author are supported by Russian Science Foundation, Project 17-11-01124. The second-named author have been partially supported by Spanish Government, Project MTM2017-83506-C2-2-P and by Principado de Asturias, Project FC -- GRUPIN -- IDI/2018/000193. The third author thanks Oviedo University, where part of this study took place, for the hospitality.

\end{document}